\documentclass{article} 

\usepackage{graphicx}
\usepackage{booktabs} 
\usepackage{array} 
\usepackage{paralist} 
\usepackage{verbatim} 
\usepackage{subfig} 

\usepackage{amssymb}
\usepackage{amsmath}
\usepackage{tensor}
\usepackage{amsthm}
\usepackage{mathtools}
\usepackage{bm}
\usepackage{mathrsfs}

\numberwithin{equation}{section}

\newtheorem{definition}{Definition}[section]

\newtheorem{corollary}{Corollary}[section]
\newtheorem{theorem}{Theorem}[section]
\newtheorem{lemma}{Lemma}[section]
\newtheorem{proposition}{Proposition}[section]
\newcommand{\be}{\begin{equation}}
\newcommand{\ee}{\end{equation}}

\begin{document}

\title{K\"ahler Finsler manifolds with curvatures bounded from below\footnote{Supported by the National Natural Science Foundation of China (No. 11871126)}}
\author{Bin Chen, Nan Li  and  Siwei Liu}

\maketitle

\begin{abstract}
We obtain a partial parallelism of the complex structure on K\"ahler Finsler manifolds. As applications, we prove Synge-Tsukamoto theorem and Bonnet-Myers theorem for positively curved K\"ahler Finsler manifolds. Moreover, we generalize a comparison theorem due to Ni-Zheng by introducing  the notion of orthogonal Ricci curvature to K\"ahler Finsler geometry.\\


\end{abstract}

\section{Introduction}

K\"ahler Finsler geometry, a natural generalization of K\"ahler geometry, was initiated by Abate-Patrizio in \cite{MG}, where the Kobayashi metric is shown to be weakly K\"ahler. Recently, there has been a surge of interest in K\"ahler Finsler geometry, especially in
its global and analytic aspects. We attend to study more global properties of K\"ahler Finsler manifolds.

The classical Synge's theorem gives simply connectedness of Riemannian manifolds with positive sectional curvature. Tsukamoto \cite{Tsu} proved a K\"ahler version under the assumption of positive holomorphic sectional curvature. The Finsler version of Synge's theorem was first derived by Auslander \cite{Au}, where the notion of flag curvature is adopted. Won introduced the notion of pseudo-K\"ahler Finsler metrics and proved a Synge type theorem. Adopting the K\"ahler notion of Abate and Patrizio, we obtain a Synge type theorem for weakly K\"ahler Finsler manifolds.

\begin{theorem}  Let $(M,G)$ be a strongly convex weakly K\"ahler Finsler manifold. Suppose it is complete and the holomorphic curvature $\mathbf{H}\geq \lambda>0$ is bounded below uniformly by a positive constant. Then $M$ is compact and simply connected.
\end{theorem}

The compactness result in the above theorem can be considered as a version of Myers's theorem. To deduce the Myers's theorem in real Finsler case, the positivity of Ricci curvature is assumed. For K\"ahler Finsler manifolds, Yin-Zhang \cite{YZ} verified a Myers's theorem for manifolds of positive bisectional curvature.
In classical K\"ahler geometry, by introducing  a notion  of curvature, namely the so-called
orthogonal Ricci curvature, Ni-Zheng \cite{NZ} proved a Myers type theorem and some comparison results. We generalize the notion of \textit{orthogonal Ricci curvature} (cf.\S6) to K\"ahler Finsler geometry and verify a Myers type theorem.
\begin{theorem} Let $(M,F)$ be a complete  strongly convex weakly K\"ahler Finsler manifold of complex dimension $n$. Suppose  the orthogonal Ricci curvature $\mathbf{Ric}^\perp\geq (2n-2)\lambda>0$, then the diameter of $M$ is at most $\pi/\sqrt{\lambda}$.
\end{theorem}

The comparison technique is widely used in Riemannian geometry.
In the real Finsler setting, Shen \cite{Shen1} first extended comparison theorems to
Finsler geometry. Later, Wu-Xin \cite{WX} proved Hessian and  Laplacian
comparison theorems under various curvature conditions. For more generalizations,
we can refer to \cite{Shen2,Oh} and references therein. In the complex Finsler realm,  by defining the Hodge Laplacian, Zhong-Zhong \cite{ZZ} and Xiao-Zhong-Qiu \cite{XZQ} obtained some
vanishing theorems and Laplacian comparison theorems on the tangent bundle. Yin-Zhang \cite{YZ} discussed the comparison theorems for the nonlinear complex Hessian. Later, Li-Qiu \cite{LQ} applied the comparison theorem to show certain K\"ahler Finsler manifold is Stein. One of the central estimates in these comparison theorems is the Laplacian comparison. The present paper  continues
investigations in this direction. We introduce the orthogonal Laplacian and  derive the following result.

\begin{theorem}Let $(M,G)$ be a complete strongly convex K\"ahler Finsler manifold of complex dimension $n$. We have the followings whenever the distance function $r$ is smooth:
\begin{itemize}
  \item[(i)] If the orthogonal Ricci curvature $\mathbf{Ric}^\perp\geq (2n-2)\lambda$, then  $$\Box^\perp r\leq (2n-2)ct_\lambda (r).$$
The precise definition of $\Box^\perp r$ and $ct_\lambda$ can be found in \S7.
  \item[(ii)] If the holomorphic curvature $\mathbf{H}\geq 4\lambda$, then   $$H(r)(J(\nabla r),J(\nabla r))\leq 2 ct_\lambda(2r)$$
      where $H(r)$ is the Hessian of $r$ and $J$ is the complex structure.
\end{itemize}
\end{theorem}
The above results provide a Finsler generalization of the results of  Liu \cite{Liu} and Ni-Zheng \cite{NZ}.  If both assumptions  $\mathbf{Ric}^\perp\geq (2n-2)\lambda$ and $\mathbf{H}\geq 4\lambda$ are satisfied, the above theorem implies a
volume comparison and a eigenvalue comparison.

\begin{corollary}
Let $(M,G,\mu)$ be a strongly convex K\"ahler Finsler $n$-manifold with
vanishing Shen curvature.  Assume $\mathbf{Ric}^\perp\geq (2n-2)K$ and $\mathbf{H}\geq 4K$ where $K$ is either $+1,0$ or $-1$.
\begin{itemize}
  \item[(i)] For $0\leq r\leq R$, it holds
$$\frac{\mathrm{Vol}^\mu_G(B_p(R))}{\mathrm{Vol}^\mu_G(B_p(r))}\leq \frac{V_K(R)}{V_K(r)}$$
where $B_p(r)$ is the geodesic ball in $M$ centered at $p$ with radius $r$, and $V_K(r)$ is the volume of the geodesic ball of radius $r$ in the complex space form.
  \item[(ii)] The first Dirichlet eigenvalue of the geodesic ball of radius $r$ centered at $p$ is bounded above by
$$\lambda_1(B_p(r))\leq \lambda_1(B(r,K))$$  where $\lambda_1(B(r,K))$ is the first Dirichlet eigenvalue of the geodesic ball of radius $r$ on
the complex space form.
\end{itemize}
\end{corollary}

The paper is organized as follows. In \S2 and \S3, some fundamental notions of complex Finsler metrics are introduced. In \S4, we verify the partial parallelism of the complex structure. In \S5 and \S6, a Synge type theorem and a Myers type theorem are proved. In \S7, the Laplacian comparison is investigated.

\section{Strongly convex complex Finsler metrics}\label{Sec2}

Let $(M,J)$ be an $n$-dimensional complex  manifold with the complex structure $J$. Suppose  $\{ z^{\alpha} \}^n_{\alpha=1}$ be a set of local complex coordinates with $z^{\alpha} = x^{\alpha} + \sqrt{-1} x^{\alpha+n}$, then $\{ x^{\alpha}, x^{\alpha+n} \}^n_{\alpha=1}$ forms a local real coordinate system.
In this coordinate system, the complex structure  has the form
\begin{equation}
J = J^i_k dx^k \otimes \frac{\partial}{\partial x^i}
\end{equation}
where
\begin{equation}\label{1eq001}
J^i_k = \left\{\begin{array}{ll}\delta^i_{k+n}, &  1\leq k \leq n,\\
 -\delta^i_{k-n}, &  n+1\leq k\leq 2n.\end{array}\right.
\end{equation}
Unless stated otherwise, we always assume that lowercase Greek letters run from $1$ to $n$ and lowercase Latin letters run from $1$ to $2n$, and the Einstein summation convention is assumed throughout this paper.

Let $T_{\mathbb{R}}M$ and  $T^{1,0} M$  be the real tangent bundle and holomorphic tangent bundle of $M$ respectively. A vector in  $T_{\mathbb{R}}M$ is denoted by $y=y^i\partial/\partial x^i$, while a vector in $T^{1,0} M$ is denoted by $v=v^\alpha \partial / \partial z^{\alpha}$.
As well known, the bundles $T^{1, 0}M$ and $T_{\mathbb{R}}M$ are isomorphic according to the bundle map
$$^o: T^{1, 0}M \rightarrow T_{\mathbb{R}}M,\hskip1cm v^o = v + \bar{v}
$$
and the inverse map
$$_o : T_{\mathbb{R}}M \rightarrow T^{1,0}M,\hskip1cm y_o = \frac12\left(y - \sqrt{-1}J y\right).
$$
Customarily, we denote $v^o = y$, $y_o = v$, and have the relation
\begin{equation}\label{2eq001}v^{\alpha} = y^{\alpha} + \sqrt{-1}y^{\alpha+n}\end{equation}

By the isomorphism, the split tangent bundles $T^{1, 0}M\backslash \{0\}$ and $T_{\mathbb{R}}M\backslash \{0\}$ can be considered as the same manifold, denoted by $\tilde M$. On the split bundle $\tilde M$, $\{z^{\alpha}, v^{\beta}\}$ is the induced complex coordinates while  $\{x^i, y^k \}$ is the local real coordinates. Thus, the complex tangent frame of  $\tilde M$ can be given as

\begin{equation*}\begin{array}{cc}
\partial_{\alpha}:= \frac{\partial}{\partial z^{\alpha}}
= \frac12\left(\frac{\partial}{\partial x^{\alpha}} - \sqrt{-1}\frac{\partial}{\partial x^{\alpha+n}}\right),
&\partial_{\bar{\alpha}}:= \frac{\partial}{\partial z^{\bar{\alpha}}}
= \frac12\left(\frac{\partial}{\partial x^{\alpha}} + \sqrt{-1}\frac{\partial}{\partial x^{\alpha+n}}\right),\\
\dot{\partial}_{\alpha}:=\frac{\partial}{\partial v^{\alpha}}
= \frac12\left(\frac{\partial}{\partial y^{\alpha}} - \sqrt{-1}\frac{\partial}{\partial y^{\alpha+n}}\right),
& \dot{\partial}_{\bar{\alpha}}:=\frac{\partial}{\partial v^{\bar{\alpha}}} = \frac12\left(\frac{\partial}{\partial y^{\alpha}} + \sqrt{-1}\frac{\partial}{\partial y^{\alpha+n}}\right).
\end{array}\end{equation*}

\begin{definition}[\cite{MG}]
A complex strongly pseudoconvex Finsler metric $F$ on a complex manifold $M$ is a continuous function
$F: T^{1, 0}M \rightarrow \mathbb{R}_{\geq 0}$ satisfying:

(i) $G = F^2$ is smooth on $\tilde M$;

(ii) $F(v) > 0$ for all $v \in \tilde{M}$;

(iii) $F(\zeta v) = |\zeta| F(v)$ for all $v \in T^{1,0}M$ and $\zeta \in \mathbb{C}$;

(iv) the Levi matrix $\left(G_{\alpha\bar{\beta}}\right)$ is positive-definite on $\tilde{M}$, where
\begin{equation}
\left(G_{\alpha\bar{\beta}}\right) = \left(\frac{\partial^2G}{\partial v^{\alpha} \partial v^{\bar{\beta}}}\right).
\end{equation}

\end{definition}

\begin{definition}[\cite{BCS}]
A real Finsler metric on a manifold $M$ is a function $F^o: T_{\mathbb{R}}M \rightarrow \mathbb{R}^+$ that satisfies the following properties:

(i) $G^o = F^{o2}$ is smooth on $\tilde{M}$;

(ii) $F^o(y) > 0$ for all $y\in \tilde{M}$;

(iii) $F^o(\lambda y) = \lambda F^o(y)$ for all $y \in T_{\mathbb{R}}M$ and $\lambda \in \mathbb{R}^+$;

(iv) the matrix $\left(g_{ij}\right)$ is positive-definite on $\tilde{M}$, where
\begin{equation}
\left(g_{ij}\right) = \frac12\left(\frac{\partial^2 G^o}{\partial y^i \partial y^j}\right).
\end{equation}
\end{definition}

Abate and Patrizio introduce the notion of \textit{strongly convex complex Finsler metrics}.
\begin{definition}[\cite{MG}]
Let $F: T^{1, 0}M \rightarrow \mathbb{R}^+$ be a  strongly pseudoconvex complex Finsler metirc. We say $F$ is strongly convex if the associated function
$F^o(y):=F(y_o)$  is a real Finsler metric.
\end{definition}

Through out this paper, we always assume $G=F^2$ is a strongly convex complex Finsler metric. For convenience,  we use the same symbol $F$ (resp. $G$) to denote the associated real Finsler metirc $F^o$ (resp. $G^o$). In other words, we shall consider $F$ or $G$ as a function of $\{z^\alpha,v^\alpha\}$ and also a function of $\{x^i,y^i\}$.\\

To start working, we need a few notations.
In complex case, we shall denote by indices like $\alpha, \bar{\beta}$ and so on the derivatives with respect to the $v$-coordinates.
The derivatives with respect to the $z$-coordinates will be denoted by indices after a semicolon.
In the real case, the notation is similar by adopting the Latin letters.  For instance, some derivatives of $G$ are denoted as follows
\begin{equation*}\begin{array}{lll}
 G_{\alpha\bar{\beta}} = \frac{\partial^2G}{\partial v^{\alpha} \partial v^{\bar{\beta}}},
&  G_{;\mu\bar{\nu}} = \frac{\partial^2 G}{\partial z^{\mu} \partial z^{\bar{\nu}}} ,
& G_{\alpha; \mu}=  \frac{\partial ^2 G}{\partial v^{\alpha} \partial z^{\mu}}; \\
G_{ij} = \frac{\partial^2G}{\partial y^i \partial y^j},
& G_{; kh} = \frac{\partial^2 G}{\partial x^k \partial x^h},
& G_{i; k} = \frac{\partial^2 G}{\partial y^i \partial x^k}.\\
\end{array}\end{equation*}


From now on, let us denote $u=Jy$. Locally, we have
\begin{equation}\label{eq2}
u=u^i \frac{\partial}{\partial x^i}= J^i_k y^k \frac{\partial}{\partial x^i}.
\end{equation}
We may state the complex Euler Theorem in the real coordinates.

\begin{lemma}\label{Lem01}
Let $H(v)$ be a complex-valued fuction on $T^{1,0}M(\cong T_{\mathbb{R}}M)$. Set
\begin{equation}
H(v) = H(y_o) = R(y) + \sqrt{-1} I(y),
\end{equation}
where $R(y)$ and $I(y)$ are real-valued functions. Then the following four conditions  are equivalent.

(i) $H(v)$ is of $(p,q)$-homogeneous;

(ii) $H(\zeta v)=\zeta^p\bar\zeta^qH(v),\forall\zeta\in \mathbb{C}$;

(iii)$H_{\alpha}v^{\alpha} = pH, H_{\bar{\alpha}}\bar{v}^{\alpha} = qH$;

(iv) $R_{k}y^k = (p+q)R,   I_{k}y^k = (p+q) I, R_{k}u^k = (q-p) I,    I_{k}u^k = (p-q)R.$

\noindent Here we adopt the abbreviations $H_\alpha=\partial H/\partial v^\alpha$ and $I_k=\partial I/\partial y^k$ and etc.
\end{lemma}
\begin{proof} The equivalence of (i)-(iii) are well known.  We shall transform (iii) into (iv).
It is easy to find
\begin{equation}v^\alpha\frac{\partial}{\partial v^\alpha}=\frac12 \left(y^i \frac{\partial}{\partial y^i} - \sqrt{-1} u^k \frac{\partial}{\partial y^k} \right).\end{equation}
Thus
\begin{align}
v^{\alpha}H_{\alpha}&= \frac12 \left(y^i \frac{\partial}{\partial y^i} - \sqrt{-1} u^k \frac{\partial}{\partial y^k} \right)  \left( R + \sqrt{-1} I \right) \nonumber \\
&= \frac12 \left( y^i R^i + u^k I_k \right) + \frac{\sqrt{-1}}{2} \left( y^i I_i - u^j R_k \right).
\end{align}
Hence, $H_{\alpha}v^{\alpha} = pH= p R + \sqrt{-1} p I$ is equivalent to
\begin{equation}\label{2eq005}
y^i R^i + u^k I_k  = 2p R, \ \ \ y^i I_i - u^j R_k  = 2p I.
\end{equation}
Similarly, $H_{\bar{\alpha}}\bar{v}^{\alpha} = qH=q R + \sqrt{-1} q I$ is equivalent to
\begin{equation}\label{2eq006}
y^i R_i - u^k I_k = 2q R,\  \ \ u^k R_k + y^i I_i = 2q I.
\end{equation}
It is clear that (\ref{2eq005}) together with (\ref{2eq006}) is equivalent to (iv).
\end{proof}

With the help of the above lemma, we can study the $J$-invariance of the \textit{real fundamental tensor}
\begin{equation}g=g_{ij}(y)dx^i\otimes dx^j.\end{equation}
This fundamental tensor can be considered as an inner product on the pullback bundle $\pi^*T_{\mathbb{R}}M$ where $\pi:\tilde M\to M$.
In other words, the inner product at $y$ can be defined as
\begin{equation}g_y(X,Y)=g_{ij}(y)X^iY^j\end{equation}
where $X=X^i\partial/\partial x^i$ and $Y=Y^j\partial/\partial x^j$. Note that $J$ can naturally act on $\pi^*T_{\mathbb{R}}M$. Thus, we can consider the $J$-invariance of $g_y$.

\begin{lemma}\label{Jinvariant}
Let $G$ be a strongly convex complex Finlser metric. Then
\begin{equation}\label{eq3} g_y(Jy, JX) = g_y(y, X), \end{equation}
for any $X \in \pi^*T_{\pi(y)}M$. Particularly, it holds
\begin{equation}\begin{array}{lr}
g_y(y,y) = g_y(Jy, Jy), & g_y(y , Jy) = 0.
\end{array}\end{equation}
\end{lemma}

\begin{proof}
Since $G(v)$ is a real-valued function $G(v) = G(v) + \sqrt{-1} \cdot 0$ and is of $(1,1)$-homogeneous, by  Lemma \ref{Lem01} (iv), we have
\begin{equation} G_k u^k = 0 \end{equation}
where $u^k = J^k_s y^s$. Differentiate the above formula with $y^i$, we have
\begin{equation} G_{ik} u^k + G_k J^k_i = 0. \end{equation}
Noting $G_{ik}=2g_{ik}$ and $G_k=2g_{jk}y^j$, one has
\begin{equation} 0=g_{ik} u^k + g_{jk}y^j J^k_i=g_{ik} J^k_jy^j + g_{jk}y^j J^k_i\end{equation}
which is equivalent to  $ g_y(Jy, JX) = g_y(y, X)$. Taking $X=y$ and $X=u$, one shall get
$g_y(y,y) = g_y(u,u)$ and $g_y(y , u)=0$.
\end{proof}

The above Lemma present the partial $J$-invariance of $g$. The following rigid result tells us $g$ is $J$-invariant if and only if the metric is Hermitian.
\begin{proposition}
Let $G$ be a strongly convex complex Finsler metric. Then $G$ is a Hermitian metric   if
$$g_y(JX,JY)=g_y(X,Y)$$
for any $X,Y \in \pi^*T_{\pi(y)}M$.
\end{proposition}
\begin{proof}
Assume $g_y(JY, JX) = g_y(Y,X)$ for any $X,Y$.
In local coordinates, we have
\begin{equation}
g_{ij} J^i_p J^j_q  = g_{pq}.
\end{equation}
Differentiating it  with $y^s$, we have
\begin{equation}\label{2eq007}
C_{ijs} J^i_p J^j_q = C_{pqs},
\end{equation}
where $C_{ijs}$ are the component of  the Cartan torsion. Since $C$ is symmetric, it turns out
\begin{equation}
C_{ijs} J^i_p J^j_q=C_{pqs}=C_{sqp} = C_{ijp} J^i_s J^j_q.
\end{equation}
Hence
\begin{equation}
C_{ijp} J^i_s J^j_q J^q_b J^p_a = C_{ijs} J^i_p  J^p_a J^j_q J^q_b
\end{equation}
that is
\begin{equation}\label{2eq008}
- C_{ipb}  J^i_s J^p_a = C_{abs}.
\end{equation}

Comparing (\ref{2eq007}) and (\ref{2eq008}), one can find
$C_{abs} = 0$ which implies $G$ is a Riemannian metric.
\end{proof}

\section{Connections and curvatures}\label{Sec3}

Let $(M, G)$ be a strongly convex complex Finsler manifold.
We shall recall the connections and curvatures of $G$ in both real and complex realm.

Denote $\hat{\mathbb{G}}^i$ the \textit{real spray coefficients} which is
\begin{equation}
\hat{\mathbb{G}}^i = \frac14 g^{il}(G_{l;k}y^k-G_{;l})
\end{equation}
where $(g^{ik}) = (g_{ij})^{-1}$.
The \textit{nonlinear connection coefficients} are defined as
\begin{equation}\hat{\mathbb{G}}^k_i=\frac{\partial \hat{\mathbb{G}}^k}{\partial y^i}.\end{equation}
Then the real horizontal frame and vertical coframe can be given by
\begin{equation}
\frac{\delta}{\delta x^i} = \frac{\partial}{\partial x^i} -\hat{\mathbb{G}}^k_i \frac{\partial}{\partial y^k},\ \ \delta y^j = dy^j + \hat{\mathbb{G}}^j_k dx^k.
\end{equation}
The \textit{Berwald connection} 1-forms $\hat{\omega}^i_j$ are
\begin{equation}
\hat{\omega}^i_j = \hat{\mathbb{G}}^i_{jk} dx^k
\end{equation}
where
\begin{equation}\hat{\mathbb{G}}^i_{jk}=\frac{\partial^2 \hat{\mathbb{G}}^i}{\partial y^j\partial y^k}=\frac{\partial \hat{\mathbb{G}}^i_j}{\partial y^k}.\end{equation}
The following give the curvature forms of the real Berwald connection
\begin{equation}
\Omega^i_j = \frac12 R^i_{jkl} dx^k \wedge dx^l  +  B^i_{jkl}  \delta y^l \wedge dx^k
\end{equation}where
\begin{equation}
R^i_{jkl} = \frac{{\delta} \hat{\mathbb{G}}^i_{jl}}{{\delta} x^k} -  \frac{{\delta}\hat{\mathbb{G}}^i_{jk} }{{\delta} x^l} + \hat{\mathbb{G}}^i_{ks} \hat{\mathbb{G}}^s_{jl} - \hat{\mathbb{G}}^s_{jk}\hat{\mathbb{G}}^i_{ls}
\end{equation} is the \textit{Riemann curvature} and
\begin{equation}B^i_{jkl} = \hat{\mathbb{G}}^i_{jkl}=\frac{\partial^3 \hat{\mathbb{G}}^i}{\partial y^j\partial y^k\partial y^l}\end{equation}
is called the \textit{Berwald curvature}.

Setting $R_{ik} = g_{si}R^s_{jkl} y^j y^l$, the \textit{flag curvature} is defined by
\begin{equation}
\mathbf{K}(y, V) = \frac{ R_{ik} V^i  V^k}{\left(g_{ij}g_{kl}-g_{ik}g_{jl}\right)y^iy^jV^kV^l}
\end{equation}
where $V = V^i \partial/\partial x^i$, and the \textit{Ricci curvature} is given by
\begin{equation}\mathbf{Ric}(y)=\frac{g^{ik}R_{ik}}{G(y)}= \sum_{j=1}^{2n-1} \mathbf{K}(y, e_i)\end{equation}
where $\{e_i\}$ is a $g_y$-orthonormal frame  with $e_{2n}=y/F$.\\

Now let us consider $(M,G)$ as a complex metric.
The Chern-Finsler nonlinear connection coefficient $\Gamma^{\alpha}_{;\beta}$ is given by
\begin{equation}
\Gamma^{\alpha}_{;\beta} := G^{\bar{\tau}\alpha}G_{\bar{\tau};\beta},
\end{equation}
where $\left(G^{\bar{\tau}\alpha}\right) = \left(G_{\beta\bar{\tau}}\right)^{-1}$. The
\textit{complex spray coefficients} $\mathbb{G}^{\alpha}$'s are
\begin{equation}\label{3eq050}
\mathbb{G}^{\alpha} = \frac12 \Gamma^{\alpha}_{;\beta} v^{\beta}.
\end{equation}
The complex horizontal frame and vertical coframe are defined as
\begin{equation}
\delta_{\mu} := \partial_{\mu} - \Gamma^{\alpha}_{;\mu} \dot{\partial}_{\alpha},\ \
\delta v^\alpha = dv^\alpha+ \Gamma^{\alpha}_{;\mu} dz^\mu.
\end{equation}
The \textit{Chern-Finsler connection} is defined by the following 1-forms
\begin{equation}
\omega^{\alpha}_{\beta} = \Gamma^{\alpha}_{\beta;\mu} dz^{\mu} + C^{\alpha}_{\beta\gamma} \delta v^{\gamma},
\end{equation}
where
\begin{equation}\begin{array}{ll}
\Gamma^{\alpha}_{\beta;\mu} := G^{\bar{\tau}\alpha} \delta_{\mu}\left(G_{\beta\bar{\tau}}\right) ,&
C^{\alpha}_{\beta\gamma} := G^{\bar{\tau}\alpha}G_{\beta\bar{\tau}\gamma}.
\end{array}\end{equation}
The curvature form of of Chern-Finsler connection is
\begin{equation}\label{3eq010}
\Omega^{\alpha}_{\beta} = \bar\partial\omega^{\alpha}_{\beta}
\end{equation}
which has four parts
\begin{equation}
\Omega^{\alpha}_{\beta} = R^{\alpha}_{\beta;\mu\bar{\nu}} dz^{\mu} \wedge dz^{\bar{\nu}} + S^{\alpha}_{\beta\mu;\bar{\nu}} \delta v^{\mu} \wedge dz^{\bar{\nu}}  + P^{\alpha}_{\beta\bar{\nu};\mu} dz^{\mu} \wedge \delta v^{\bar{\nu}} + Q^{\alpha}_{\beta\mu\bar{\nu}} \delta v^{\mu} \wedge \delta v^{\bar{\nu}}.
\end{equation}
The \textit{holomorphic curvature} $\mathbf{H}$ is defined as
\begin{equation}\label{3eq002}
\mathbf{H}(v) = \frac{2}{G^2} R_{\alpha\bar{\beta}; \mu\bar{\nu}} v^{\alpha} v^{\bar{\beta}} v^{\mu}v^{\bar{\nu}}
\end{equation}
where $R_{\alpha\bar{\beta}; \mu\bar{\nu}} = G_{\sigma\bar{\beta}}R^{\sigma}_{\alpha; \mu\bar{\nu}}$.

As the end of this section, we state the K\"ahler conditions defined by Abate-Patrizio.
\begin{definition}[\cite{MG}, \cite{CB}]
Let $G$ be a strongly pseudoconvex complex Finsler metric.

(i) $G$ is called a  K\"ahler Finsler metric if and only if $\Gamma^{\alpha}_{\beta;\mu} = \Gamma^{\alpha}_{\mu;\beta}$;

(ii) $G$ is called a weakly K\"ahler Finsler metric if $G_{\alpha}\left(\Gamma^{\alpha}_{\beta;\mu}- \Gamma^{\alpha}_{\mu;\beta}\right) v^{\mu} = 0$.
\end{definition}

For a weakly K\"ahler Finsler metric,
Li-Qiu-Zhong discovered that the holomorphic curvature coincides with the flag curvature of the holomorphic flag.
\begin{theorem}[\cite{LQZ}]\label{holo}
Let $G$ be is a strongly convex weakly K\"ahler Finsler metric, then
$\mathbf{K}(y,Jy)=\mathbf{H}(v)$ where $v=y_o$.
\end{theorem}
We shall point out that  the holomorphic curvature defined in this paper is twice of that in \cite{LQZ}. Thus the above theorem is slightly different in \cite{LQZ}.

\section{Parallelism of the complex structure}\label{Sec5}
For a K\"ahler Hermitian metric, the complex structure is parallel with respect to the Riemann connection. Let us consider the weakly K\"ahler Finsler case in this section.

According to  Proposition 2.6.2 in  \cite{MG} or Lemma 3.1 in \cite{XZ},  a strongly convex complex Finsler metric   is a weakly K\"ahler  if and only if the complex spray and real spray satisfy
\begin{equation}\label{5eq002}
\mathbb{G}^{\beta} = \hat{\mathbb{G}}^{\beta} + \sqrt{-1}\hat{\mathbb{G}}^{\beta+n}
\end{equation}
in the complex coordinate system.  We shall calculate the covariant derivative of $J$ in this coordinate system.

Since $\mathbb{G}^{\alpha}$ are of $(2,0)$-homogeneous, by (\ref{1eq001}) in Lemma \ref{Lem01}, we have
\begin{equation}\label{5eq006}  \begin{array}{lcl}
\hat{\mathbb{G}}^{\alpha}_k u^k = -2 \hat{\mathbb{G}}^{\alpha+n} & \text{and} &
\hat{\mathbb{G}}^{\alpha+n}_k u^k = 2\hat{\mathbb{G}}^{\alpha}. \end{array}
\end{equation}
Recall that $J = J^i_k dx^k \otimes \frac{\partial}{\partial x^i}$ has the form
\begin{equation} \left\{ \begin{array}{l}
J^i_{\alpha} = \delta^i_{\alpha+n},\\
J^i_{\alpha +n} = - \delta^i_{\alpha} \end{array} \right.
\end{equation}
in this coordinate system.
Thus, (\ref{5eq006}) can be read as
\begin{equation}\label{parallel0} \hat{\mathbb{G}}^i_k u^k = 2J^i_k\hat{\mathbb{G}}^k.\end{equation}
Taking the derivatives with respect to $y^s$ ,  one has
\begin{equation}\label{parallel00}
\hat{\mathbb{G}}^i_{sk} u^k + \hat{\mathbb{G}}^i_k J^k_s = 2J^i_k\hat{\mathbb{G}}^k_s.
\end{equation}

The covariant derivative of $J$ with respect to the real Berwald connection is defined by
\begin{align}
J^i_{k|l} &:= \frac{\delta J^i_k}{\delta x^l} +  J^s_k \hat{\mathbb{G}}^i_{sl} - J^i_s \hat{\mathbb{G}}^s_{kl}.
\end{align}
 Since $J^i_k$ are constant in the complex coordinate system, locally it holds
\begin{align}
J^i_{k|l}&=J^s_k \hat{\mathbb{G}}^i_{sl} - J^i_s \hat{\mathbb{G}}^s_{kl}.
\end{align}
Contracting with $y^k$ and $y^l$, we reach
\begin{align}
J^i_{k|l} y^k y^l& = J^s_k \hat{\mathbb{G}}^i_{sl}y^ly^k - J^i_s \hat{\mathbb{G}}^s_{kl}y^ly^k  \nonumber\\
&=J^s_k \hat{\mathbb{G}}^i_s y^k - J^i_s \hat{\mathbb{G}}^s_k y^k = \hat{\mathbb{G}}^i_s u^s - 2 J^i_s \hat{\mathbb{G}}^s.
\end{align}
By (\ref{parallel0}) and $y^k_{|l}=0$, it turns out

\begin{lemma}\label{lemJ}
On a  strongly convex weakly K\"ahler Finsler manifold, it holds
\begin{equation}
\left(\nabla_{y^{\mathcal{H}}}J\right)y = \nabla_{y^{\mathcal{H}}}\left(Jy\right)=0.
\end{equation}
where $y^{\mathcal{H}} = y^i \delta/\delta x^i$ and $\nabla $ is the covariant differential with respect to the Berwald connection.
\end{lemma}
Since the real Berwald connection and real Chern connection only differ by a term of Landsberg curvature, the above lemma is also true for the real Chern connection.
\section{Synge-Tsukamoto theorem}

In this section, we shall prove the Synge-Tsukamoto theorem for weakly K\"ahler Finsler manifolds.

Let us recall the theory of geodesics in real Finsler geometry (cf. \cite{BCS}). Consider any unit speed geodesic $\gamma(t)(0\leq t\leq d)$ with velocity field $T=\dot\gamma$.
Set the linear covariant derivative with reference vector $T$ as
\begin{equation}D^T_{\frac{\partial}{\partial x^j}}\frac{\partial}{\partial x^i}=\hat{\mathbb{G}}^k_{ij}(T)\frac{\partial}{\partial x^k}.\end{equation}
Then $D^T_TT=0$.
 Suppose $U$ is a variation field along $\gamma$, then the second variation of arc length is
\begin{align}\label{2ndvar}
L''=I_\gamma(U_\perp,U_\perp)+g_T(D^T_UU,T)|_{t=0}^d.
\end{align}
where $U_\perp=U-g_T(U,T)T$ is the $g_T$-orthogonal component with respect to $T$, and $I_\gamma$ is the \textit{index form} of $\gamma$
\begin{align}
I_\gamma(W,W)&=\int_0^d\Big[g_T(D_T^TW,D_T^TW)-g_T(R(W,T)T,W)\Big]dt
\end{align}
where $W$ is any vector field along $\gamma$. Particularly, if $g_T(T,W)=0$, we have
\begin{align}
I_\gamma(W,W)&=\int_0^d\Big[g_T(D_T^TW,D_T^TW)-g_T(W,W)\cdot \mathbf{K}(T,W)\Big]dt
\end{align}

Denote $V=JT$. According to Lemma \ref{Jinvariant} and Lemma \ref{lemJ}, we have
\begin{equation}\label{TV}g_T(T, V)=0,\ \  g_T(V,V)=1,\ \ D^T_TV=0.\end{equation}

\begin{theorem}  Let $(M,G)$ be a strongly convex weakly K\"ahler Finsler manifold. Suppose it is complete and the holomorphic curvature $\mathbf{H}\geq \lambda>0$ is bounded below uniformly by a positive constant. Then $M$ is compact and simply connected.
\end{theorem}
\begin{proof} Let us fist verify the compactness of $M$. Consider any unit speed geodesic $\gamma(t), 0\leq t\leq d$ with length $d=\pi/\sqrt{\lambda}$. Set
\begin{equation}s_\lambda(t)=\sin(\sqrt{\lambda}t)\end{equation}
and
\begin{equation}W(\gamma(t))=s_\lambda(t)V(\gamma(t)).\end{equation}
Then
\begin{equation}D^T_TW=s_\lambda'(t) V=\sqrt{\lambda}\cos(\sqrt{\lambda}t)V.\end{equation}
According to Theorem \ref{holo}, we have
\begin{equation}\mathbf{K}(T,W)=\mathbf{K}(T,V)=\mathbf{H}(V_o).\end{equation}
Thus, it holds
\begin{align}
I_\gamma(W,W)&=\int_0^d[g_T(D_T^TW,D_T^TW)-g_T(W,W)\mathbf{K}(T,W)]dt\nonumber\\
&=\int_0^d[g_T(D_T^TW,D_T^TW)-g_T(W,W)\mathbf{H}(T_o)]dt\nonumber\\
&\leq \int_0^d[(s_\lambda'(t))^2-\lambda(s_\lambda(t))^2]dt=\frac{1}{2}\sqrt{\lambda}\sin(2\sqrt{\lambda}d)=0.
\end{align}
A standard argument shows that the geodesic $\gamma$ must contain conjugate points, which implies the diameter of $M$ is bounded from above
\begin{equation} \mathrm{diam}(M)\leq d= \frac{\pi}{\sqrt{\lambda}}.\end{equation}
Therefore, $M$ is compact.

 Next, we show that $M$ is simply connected. Suppose the fundamental group $\pi(M)$ is nontrivial. Picking a nontrivial free homotopy class,  according to Theorem 8.7.1 in \cite{BCS}, there exists a shortest smooth closed geodesic $\gamma(t)(0\leq t\leq l)$ in this class.  The definition of vector fields $T$ and $V$ are the same as above. Produce a variation of $\gamma(t)$ such that the resulting variation vector field is $V(\gamma(t))$. Hence, the second variation of arc length reduces to
 \begin{equation}L''=-\int_0^l\mathbf{H}(T_o)dt\leq \lambda l<0\end{equation}
 which implies $\gamma$ is not minimal. This leads a contradiction.
\end{proof}

Following the argument in \cite{NZ}, we have the following result.
\begin{theorem}
Let $(M,G)$ be a compact strongly convex weakly K\"ahler Finsler manifold with positive holomorphic curvature. Then any holomorphic isometry of $M$ must have at least one fixed point.
\end{theorem}
\begin{proof} Let $d(\cdot,\cdot)$ be the distance function of $G$. Assume that there exists such a holomorphic isometry $f: M\to M$ with no fixed point.  Then
the continuous function $d(x, f(x))$ can attain its positive minimum, saying
$$d(p,f(p))=\min d(x,f(x))=l>0.$$ Let $\gamma$ be a minimal geodesic joining
$p$ to $f(p)$ with $\gamma(0)=p$ and $\gamma(l)=f(p)$.

Since $f$ is an isometry and $\gamma$ is a minimal geodesic, we have
\begin{align*}
l\leq d(\gamma(t), f(\gamma(t)))&\leq d(\gamma(t), f(p))+d(f(p), f(\gamma(t))).\\
&= d(\gamma(t), f(p))+d(p, \gamma(t))=d(p,f(p))=l.
\end{align*}
Thus $\gamma\cup f(\gamma)$ is a geodesic which is smooth at $f(p)$ and hence $df(\dot \gamma(0))=\dot\gamma(l)$.

Again, set $T=\dot\gamma$ and $V=JT$. Since $f$ is holomorphic, we have $df(V)=J(df(T))$ and hence

\begin{equation} V(l)=J(T(l))=J(df(T(0)))=df(J(T(0)))=df(V(0)).\end{equation}
 Consider the variation $\tilde \gamma(t,s)=\exp_{\gamma(t)}(sV(t))$.
Denoting $\tilde \gamma_s$ the variational field, we have $\tilde \gamma_s(t,0)=V(t)$. Since $f$ is an isometry, $f(\tilde \gamma(0,s))$ is also a geodesic with the initial velocity
\begin{equation}df(\tilde \gamma_s(0,0))=df(V(0))=V(l)=\tilde \gamma_s(l,0).\end{equation}
It implies the geodesic $\tilde \gamma(l,s)$ is just $f(\tilde \gamma(0,s))$ and hence $\tilde \gamma_s(l,s)=df(\tilde \gamma_s(0,s))$.

Since $f$ is an isometry, together with $df(T(0))=T(l)$ and $df(\tilde \gamma_s(0,s))=\tilde \gamma_s(l,s)$, it must hold
\begin{equation}g_T(D^T_{\tilde\gamma_s}{\tilde\gamma_s},T)|_{t=0}^l=0.\end{equation}
Thus, the second variation of arc length (\ref{2ndvar}) becomes
\begin{equation}L''=I_\gamma(V,V)=-\int_0^l\mathbf{K}(T,V)dt=-\int_0^l\mathbf{H}(V_o)dt<0.\end{equation}
This contradicts to
\begin{equation}d(\tilde \gamma(0,s), \tilde \gamma(l,s))= d(\tilde \gamma(0,s), f(\tilde \gamma(0,s)))\geq l.\end{equation}
and $\gamma(t)=\tilde\gamma(t,0)$ is minimizing.
\end{proof}
\section{Bonnet-Myers theorem}
Following Ni-Zheng, we introduce the \textit{orthogonal Ricci curvature} for strongly convex weakly K\"ahler Finsler metrics as
\begin{equation} \mathbf{Ric}^\perp(y)=\mathbf{Ric}(y)- \mathbf{K}(y,Jy)=\mathbf{Ric}(y)- \mathbf{H}(y_o).\end{equation}

\begin{theorem} Let $(M,F)$ be a complete  strongly convex weakly K\"ahler Finsler metric of complex dimension $n$. Suppose $\mathbf{Ric}^\perp\geq (2n-2)\lambda>0$, then the diameter of $M$ is at most $\pi/\sqrt{\lambda}$.
\end{theorem}
\begin{proof} The proof is similar to the first part of Theorem 5.1.
Let  $\gamma$ be a unit-speed geodesic such that $\gamma(0)=p$ and $\gamma(d)=q$ where $d=\pi/\sqrt{\lambda}$ is the length of the geodesic. Set $T=\dot\gamma$ and $V=JT$. Pick a $g_T$-orthonormal frame $\{E_i\}$ at $q$ such that
$$E_{2n}=T|_q,\ \ \  E_{2n-1}=V|_q.$$
Let $\{E_i(t)\}$ be the parallel transportation of $E_i$ along $\gamma$, i.e. $D^T_TE_i=0$. By (\ref{TV}), we have $E_{2n}(t)=T(t)$ and $E_{2n-1}(t)=V(t)$. Moreover $E_i(t)(i\leq 2n-2)$ are $g_T$-orthogonal to $T$  and $V$.  Set
$$W_i= s_\lambda(t) E_i(t)=\sin(\sqrt{\lambda}t)E_i(t).$$
We have
\begin{align*}
\sum_{i=1}^{2n-2}I_\gamma(W_i,W_i)&=\sum_{i=1}^{2n-2}\int_0^d[g_T(D_T^TW_i,D_T^TW_i)-g_T(W_i,W_i)\mathbf{K}(T,W_i)]dt\\
&= \sum_{i=1}^{2n-2}\int_0^d[(s'_\lambda(t))^2-(s_\lambda(t))^2\mathbf{K}(T,E_i)]dt\\
&= \int_0^d[(2n-2)(s'_\lambda(t))^2-\mathbf{Ric}^\perp(T)(s_\lambda(t))^2]dt\\
&\leq (2n-2)\int_0^d[(s'_\lambda(t))^2-\lambda(s_\lambda(t))^2]dt=0.
\end{align*}
Then, there exists $i_0$ such that $I_\gamma(W_{i_0},W_{i_0})\leq 0$. A standard argument shows that $\gamma$ contains conjugate points. Therefore, $\mathrm{diam}(M)\leq d=\pi/\sqrt{\lambda}$.
\end{proof}

\section{Laplacian comparison}

\setcounter{equation}{0}

Let $(M,G)$ be a complete strongly convex weakly K\"ahler Finsler  manifold of complex dimension $n$.
We shall consider the  Laplacian of the distance function in this section.

For any smooth function $f$ on $M$, its \textit{gradient} is defined by
\begin{equation}\nabla f=\mathcal{L}^{-1}(df)\end{equation}
where $\mathcal{L}^{-1}$ is the inverse of Legendre transformation
\begin{equation}\mathcal{L}(X)=\left\{\begin{array}{cc}
                                        g_X(X,\cdot), & X\not=0, \\
                                        0, & X=0.
                                      \end{array}
\right.\end{equation}
Let $\mathcal{U}_f=\{x\in M|df(x)\not=0\}$. As $D^T$ introduced in \S5, we define a linear connection $D^{\nabla f}$ on $\mathcal{U}_f$ by setting
\begin{equation}D^{\nabla f}_{\frac{\partial}{\partial x^j}}\frac{\partial}{\partial x^i}=\hat{\mathbb{G}}^k_{ij}(\nabla f)\frac{\partial}{\partial x^k}.\end{equation}
The \textit{Hessian} of $f$ can be defined as (cf. \cite{WX})
\begin{equation}H(f)(X,Y)=XY(f)-D^{\nabla f}_XY(f)=\left(\frac{\partial^2f}{\partial x^i\partial x^j}-\hat{\mathbb{G}}^k_{ij}(\nabla f)\frac{\partial f}{\partial x^k}\right)X^iY^j.\end{equation}
It should be noted that the connection $D^{\nabla f}$  defined here is different from that
in \cite{WX}. In that case the Chern connection is used. Fortunately,  the Hessian defined here is the same as that in \cite{WX}. In fact,
the difference between the real Berwald connection and real Chern connection is the Landsberg curvature $L^i_{jk}$, and hence $L^k_{ij}(\nabla f)\frac{\partial f}{\partial x^k}=0$. Moreover, the following identity is also true:
\begin{equation}H(f)(X,Y)=g_{\nabla f}(D^{\nabla f}_X\nabla f,Y),\ \ \ \forall X,Y\in T_{\mathbb{R}}M|_{\mathcal{U}_f}.\end{equation}
The $g_{\nabla f}$-trace of $H(f)$ can be considered as some second order differential operator acting on $f$. Let us denote the trace by
\begin{equation}\Box f:=\mathrm{tr}_{g_{\nabla f}}H(f)=\sum_{i=1}^{2n}H(f)(E_i,E_i)\end{equation}
where $\{E_i\}$ is a local $g_{\nabla f}$-orthonormal frame  on $\mathcal{U}_f$.\\

 Denote  $r(x)=d(p,x)$ be  the distance function from a fixed point $p\in M$. It is well known that $r$ is smooth on $M\backslash\{p\}$ away from the cut points of $p$, and $G(\nabla r)=1$  (cf. \cite{Shen1} or \cite{WX}).
Following Ni and Zheng (\cite{NZ}), we define the $\Box^\perp$-operator as
\begin{equation}\Box^\perp r:=\Box r-H(r)(\nabla r,\nabla r)-H(r)(J(\nabla r),J(\nabla r)).\end{equation}
Denote $\nabla r=T$ and $J(\nabla r)=V$ for abbreviation.  We shall estimate $\Box r$.

Assume $\gamma$ is a unit-speed geodesic
without a conjugate point up to distance $r$ from $p$ and $\gamma(r)=q$. Then $\dot\gamma=T|_\gamma$. Thus
\begin{equation}H(r)(T,T)=g_T(D^T_TT,T)=0,\ \ H(r)(T,V)=g_T(D^T_TT,V)=0.\end{equation}
Let $\{E_i\}$  be a $g_{\nabla r}$-orthonormal frame at $q$ such that
$$E_{2n}=T|_q,\ \ \  E_{2n-1}=JT|_q=V|_q.$$
Let $\{J_i\}$ be   the Jacobi fields along $\gamma$ with
$$J_i(0)=0,J_i(r)=E_i.$$
According to (4.1) in \cite{WX},  we have
$$H(r)(E_i,E_i)|_q=I_\gamma(J_i,J_i)=I_\gamma(J_i^\perp,J_i^\perp)$$
where $J_i^\perp=J_i-g_T(T,J_i)$.

Let $E_i(t)$ be the parallel transportation of $E_i$ along $\gamma$. Set
\begin{equation}s_\lambda(t)=\left\{\begin{array}{lc}
                                      \sin(\sqrt{\lambda}t), & \lambda>0, \\
                                      t, & \lambda=0, \\
                                      \sinh(\sqrt{-\lambda}t), & \lambda<0.
                                    \end{array}
\right.\end{equation} Produce
$$W_i=\frac{s_\lambda(t)}{s_\lambda(r)}E_i(t).$$
Since $J_i(t)$ and $W_i(t)$ have the same boundary value, by the index lemma, we have
\begin{align*}
\Box^\perp r|_q&=\sum_{i=1}^{2n-2}H(r)(E_i,E_i)|_q=\sum_{i=1}^{2n-2}I_\gamma(J_i,J_i)\leq \sum_{i=1}^{2n-2}I_\gamma(W_i,W_i)\\
&=\frac{1}{(s_\lambda(r))^2}\int_0^r\Big((2n-2)(s'_\lambda(t))^2-\mathbf{Ric}^\perp(T)(s_\lambda(t))^2\Big)dt.
\end{align*}
By setting
\begin{equation}ct_{\lambda}(t)=\left\{\begin{array}{lc}
                                        \sqrt{\lambda}\cot (\sqrt{\lambda}t), & \lambda>0, \\
                                        1/t, & \lambda=0,\\
                                         \sqrt{-\lambda}\coth (\sqrt{-\lambda}t) & \lambda<0, \\
                                      \end{array}
\right.\end{equation}
one can immediately obtain the following theorem.
\begin{theorem}Let $(M,G)$ be a complete strongly convex weakly K\"ahler Finsler manifold of complex dimension $n$. If $\mathbf{Ric}^\perp\geq (2n-2)\lambda$, then we have the following inequality whenever $r$ is smooth $$\Box^\perp r\leq (2n-2)ct_\lambda (r).$$
\end{theorem}

Now we shall estimate  $H(r)(V,V)$. In \cite{YZ}, the complex Hessian of $r$ is introduced. Let us  define
\begin{equation} r_{11}:=T_o^\alpha T_o^\beta \frac{\partial^2 r}{\partial z^\alpha\partial z^\beta}-2\mathbb{G}^\alpha(T_o) \frac{\partial r}{\partial z^\alpha}\end{equation}
where
$$T_o=\frac{1}{2}(T-\sqrt{-1}JT)=T_o^\alpha\frac{\partial}{\partial z^\alpha}.$$
We point out that $r_{11}$ has the same meaning in \cite{YZ} (cf. Theorem 4.4 therein).

\begin{lemma}Assume $G$ is  strongly convex and weakly K\"ahler. It holds
\begin{equation}H(r)(V,V)=-4r_{11}\end{equation}
\end{lemma}
\begin{proof} Writting $T=T^i\partial/\partial x^i$ and $V=V^i\partial/\partial x^i$, it turns out
\begin{equation}T^\alpha=V^{\alpha+n},\ \ T^{\alpha+n}=-V^\alpha,\ \ T_o^\alpha=T^\alpha+\sqrt{-1}T^{\alpha+n}.\end{equation}
Recalling (\ref{5eq002}), we have
$$\mathbb{G}^{\alpha}(T_o) = \hat{\mathbb{G}}^{\alpha}(T) + \sqrt{-1}\hat{\mathbb{G}}^{\alpha+n}(T),$$
thus
\begin{align}
2\mathbb{G}^\alpha(T_o) \frac{\partial r}{\partial z^\alpha}&= (\hat{\mathbb{G}}^{\alpha}(T) + \sqrt{-1}\hat{\mathbb{G}}^{\alpha+n}(T))(\frac{\partial r}{\partial x^\alpha}-\sqrt{-1}\frac{\partial r}{\partial x^{\alpha+n}})\nonumber\\
&=\hat{\mathbb{G}}^{k}(T)\frac{\partial r}{\partial x^k}-\sqrt{-1}\hat{\mathbb{G}}^{i}(T)J^k_i\frac{\partial r}{\partial x^k}.
\end{align}
One can immediately  get
\begin{align*}
T_o^\alpha T_o^\beta &=(T^\alpha+\sqrt{-1}T^{\alpha+n})(T^\beta+\sqrt{-1}T^{\beta+n})\\
&=T^\alpha T^\beta-T^{\alpha+n}T^{\beta+n}+\sqrt{-1}(T^{\alpha+n}T^\beta+T^\alpha T^{\beta+n})
\end{align*}
and
\begin{align*}
4 \frac{\partial^2 r}{\partial z^\alpha\partial z^\beta}&= \frac{\partial^2 r}{\partial x^\alpha\partial x^\beta}-\frac{\partial^2 r}{\partial x^{\alpha+n}\partial x^{\beta+n}}-\sqrt{-1}\left(\frac{\partial^2 r}{\partial x^{\alpha+n}\partial x^\beta}+\frac{\partial^2 r}{\partial x^\alpha \partial x^{\beta+n}}\right).
\end{align*}
Thus
\begin{align}
4 T_o^\alpha T_o^\beta\frac{\partial^2 r}{\partial z^\alpha\partial z^\beta}&= T^iT^j \frac{\partial^2 r}{\partial x^i\partial x^j}-V^iV^j\frac{\partial^2 r}{\partial x^i\partial x^j}-2\sqrt{-1}V^iT^i\frac{\partial^2 r}{\partial x^i\partial x^j}.
\end{align}
Hence
\begin{align}\label{r11-1}
4r_{11}&=4 T_o^\alpha T_o^\beta\frac{\partial^2 r}{\partial z^\alpha\partial z^\beta}-8\mathbb{G}^\alpha(T_o) \frac{\partial r}{\partial z^\alpha}\nonumber\\
&=T^iT^j \frac{\partial^2 r}{\partial x^i\partial x^j}-V^iV^j\frac{\partial^2 r}{\partial x^i\partial x^j}-4\hat{\mathbb{G}}^{k}(T)\frac{\partial r}{\partial x^k}\nonumber\\
&\hskip2cm -2\sqrt{-1}\left(V^iT^i\frac{\partial^2 r}{\partial x^i\partial x^j}-2\hat{\mathbb{G}}^{i}(T)J^k_i\frac{\partial r}{\partial x^k}\right)\nonumber\\
&=H(r)(T,T)-H(r)(V,V)+ \left(T^iT^j\hat{\mathbb{G}}^{k}_{ij}(T) -V^iV^j\hat{\mathbb{G}}^{k}_{ij}(T) -4\hat{\mathbb{G}}^{k}(T)\right)\frac{\partial r}{\partial x^k}\nonumber\\
&-2\sqrt{-1}\left(H(r)(T,V)+ \left(T^iV^j\hat{\mathbb{G}}^{k}_{ij}(T)-2\hat{\mathbb{G}}^{i}(T)J^k_i\right)\frac{\partial r}{\partial x^k}\right).
\end{align}
By $T^i\hat{\mathbb{G}}^{k}_{ij}(T)=\hat{\mathbb{G}}^{k}_{j}(T)$  and (\ref{parallel0}), we have
\begin{equation}\label{r11-2}T^iV^j\hat{\mathbb{G}}^{k}_{ij}(T)=V^j\hat{\mathbb{G}}^{k}_{j}(T)=2\hat{\mathbb{G}}^{i}(T)J^k_i.\end{equation}
According to (\ref{parallel00}), we have
\begin{align}\label{r11-3}
V^iV^j\hat{\mathbb{G}}^{k}_{ij}(T)&=2J^k_i\hat{\mathbb{G}}^{i}_{j}(T)V^j-\hat{\mathbb{G}}^{k}_{i}(T)J^i_jV^j\nonumber\\
&=4J^k_i\hat{\mathbb{G}}^{j}(T)J_j^i+\hat{\mathbb{G}}^{k}_{i}(T)T^i\nonumber\\
&=-4\hat{\mathbb{G}}^{k}(T)+2\hat{\mathbb{G}}^{k}(T)=-2\hat{\mathbb{G}}^{k}(T)=-T^iT^j\hat{\mathbb{G}}^{k}_{ij}(T).
\end{align}
Substituting (\ref{r11-2}), (\ref{r11-3})  and  $H(r)(T,T)=H(r)(T,V)=0$ into (\ref{r11-1}), we reach $4r_{11}=-H(r)(V,V)$.
\end{proof}

\begin{theorem}Let $(M,G)$ be a complete strongly convex K\"ahler Finsler manifold of complex dimension $n$.
If $\mathbf{H}\geq 4\lambda$, then we have the following inequality whenever $r$ is smooth  $$H(r)(J(\nabla r),J(\nabla r))\leq 2 ct_\lambda(2r).$$
\end{theorem}
\begin{proof}
Define $f(t)=-r_{11}(\gamma(t))$ along the radial geodesic $\gamma$. According to the proof of Theorem 4.4 in \cite{YZ}, it holds
\begin{equation}4f^2(t)+f'(t)\leq -\frac{1}{4}\mathbf{H}(T_o)\end{equation}
and
\begin{equation}\lim_{t\to 0^+}tf(t)=\frac{1}{4}\end{equation}
since $G$ is a K\"ahler Finsler metric. By the assumption  $\mathbf{H}\geq 4\lambda$, one can get
\begin{equation}f(t)\leq \frac{1}{2}ct_{\lambda}(2t).\end{equation}
Therefore, $H(f)(V,V)|_{\gamma(t)}=4f\leq 2 ct_{\lambda}(2t)$.
\end{proof}

\begin{corollary}Let $(M,G)$ be a complete strongly convex K\"ahler Finsler manifold of complex dimension $n$. If $\mathbf{Ric}^\perp\geq (2n-2)\lambda$ and $\mathbf{H}\geq 4\lambda$, then we have the following inequality whenever $r$ is smooth
\begin{equation}\label{boxestimate}\Box r\leq (2n-2)ct_\lambda (r)+2ct_\lambda (2r).\end{equation}
\end{corollary}

Using the above Laplacian comparison, one can  follow the arguments in Section 5 of \cite{YZ} and obtain the following volume comparison and eigenvalue comparison.

\begin{corollary}Let $(M,G)$ be a complete strongly convex K\"ahler Finsler $n$-manifold  with arbitrary measure $\mu$. Assume $\mathbf{Ric}^\perp\geq (2n-2)K$ and $\mathbf{H}\geq 4K$ where $K$ is either $+1,0$ or $-1$. If the Shen curvature vanishes, then for $0\leq r\leq R$ it holds
$$\frac{\mathrm{Vol}^\mu_G(B_p(R))}{\mathrm{Vol}^\mu_G(B_p(r))}\leq \frac{V_K(R)}{V_K(r)}$$
where $B_p(r)$ is the geodesic ball centered at $p$ with radius $r$, and $V_K(r)$ is the volume of the geodesic ball of radius $r$ in the complex space form.
\end{corollary}

\begin{corollary}
Let $(M,G,\mu)$ be a strongly convex K\"ahler Finsler $n$-manifold with
vanishing Shen curvature.  Assume $\mathbf{Ric}^\perp\geq (2n-2)K$ and $\mathbf{H}\geq 4K$ where $K$ is either $+1,0$ or $-1$. Then the first Dirichlet eigenvalue of the geodesic ball of radius $r$ centered at $p$ is bounded above by
$$\lambda_1(B_p(r))\leq \lambda_1(B(r,K))$$  where $\lambda_1(B(r,K))$ is the first Dirichlet eigenvalue of the geodesic ball of radius $r$ on
the complex space form.
\end{corollary}


\noindent Bin Chen\\
School of Mathematical Sciences, Tongji University\\
Shanghai, China, 200092\\
chenbin@tongji.edu.cn\\

\noindent Nan Li\\
School of Mathematical Sciences, Tongji University\\
Shanghai, China, 200092\\
2130901@tongji.edu.cn\\

\noindent Siwei Liu\\
School of Mathematical Sciences, Tongji University\\
Shanghai, China, 200092\\
carrot98@163.com

\end{document}